\documentclass[12pt]{article}

\usepackage[english]{babel}
\usepackage[cp1251]{inputenc}
\usepackage[T2A]{fontenc}

\usepackage{amssymb}
\usepackage{amsmath}
\usepackage{amsthm} 
\usepackage{mathtext}
\usepackage{amsfonts}

\renewcommand{\leq}{\leqslant}
\renewcommand{\geq}{\geqslant}

\newcommand{\refpar}{Sect.~}

\newcommand{\lda}{\lambda}
\newcommand{\Wo}{{\raisebox{0.2ex}{\(\stackrel{\circ}{W}\)}}{}}

\newtheorem{theorem}{Theorem}

\newtheorem{proposition}{Proposition}

\theoremstyle{definition}
\newtheorem{definition}{Definition}
\newtheorem{remark}{Remark}

\newenvironment{enbibliography}{\vspace{-0.5cm}\begin{thebibliography}}{\end{thebibliography}}

\begin{document} 
\title{On spectral asymptotics of the 
Neumann problem for the Sturm-Liouville 
equation with self-similar generalized 
Cantor type weight}
\author{N.~V.~Rastegaev \\ \small{Chebyshev Laboratory, St. Petersburg State University}
\\ \small{14th Line, 29b, Saint Petersburg, 199178 Russia} 
\\ \small{rastmusician@gmail.com}}
\maketitle



\section{Introduction}

We generalize some results of \cite{SV} and 
\cite{VSh3}, on the spectral asymptotics of the problem
\begin{gather}\label{eq:1.1}
    -y''-\lambda\rho y=0,\\ \label{eq:1.2}
    y'(0)=y'(1)=0,
\end{gather}
where the weight measure $\rho$ is a distributional derivative of a self-similar generalized Cantor type function
(in particular, $\rho$ is singular with respect to the Lebesgue measure). 
\begin{remark}
It is well known, that the change of the boundary conditions causes rank two perturbation of the quadratic form. It follows from the general variational theory 
(see \cite[\S10.3]{BS2}) that counting functions of the eigenvalues of boundary-value problems, related to the same equation, but different boundary conditions, cannot differ by more than 2.
\end{remark}
The problem of the eigenvalues asymptotic behavior for this problem goes back to the works of
M.~G.~Krein (see, for example, \cite{K}).

From \cite{BS1} it follows that if the measure $\rho$ contains absolutely
continuous component, its singular component does not influence
the main term of the spectral asymptotic.

In the case of singular measure $\rho$ it could be seen from \cite{B} that the counting function $N:(0,+\infty)\to\mathbb N$ of eigenvalues of boundary value problem for the operator $(-1)^ly^{(2l)}$ admits the estimate $o(\lda^{\frac{1}{2l}})$ instead of the usual asymptotics $N(\lda)\sim C\lda^{\frac{1}{2l}}$ in the case of measure containing a regular component.
Better lower bounds for eigenvalues were obtained for some special classes of measures in \cite{B}.

Exact power exponent $D$ of the counting function $N(\lda)$ in
the case of self-similar measure $\rho$ was established in \cite{F} (see also earlier works \cite{HU} and \cite{MR} for partial results,
concerning the classical Cantor ladder).

It is shown in \cite{SV} and \cite{KL} that the eigenvalues counting function of problem \eqref{eq:1.1}, \eqref{eq:1.2} has the asymptotics
\begin{equation}\label{eq:mes_asymp}
	N(\lambda)=\lambda^D\cdot\bigl(s(\ln\lambda)+o(1)\bigr),
	\qquad \lambda\to+\infty,
\end{equation}
where $D\in(0,\frac{1}{2})$ and $s$ is a continuous periodic function, dependent on the choice of the weight $\rho$.
In the case of non-arithmetic type of self-similarity (see Definition \ref{arithm} below) of the Cantor ladder primitive for $\rho$, the function $s$ degenerates into constant. In the case of arithmetic self-similarity it has a period $\nu$, which depends on the parameters of the ladder.

In the paper \cite{Naz} this result is generalized to the case of higher even order differential operator. Also it is conjectured in \cite{Naz}, that the function $s$ is not constant for arbitrary non-constant weight with arithmetically self-similar primitive.

The paper \cite{VSh1} gives computer-assisted proof of this conjecture in the simplest case, when the generalized primitive of weight $\rho$ is a classical Cantor ladder.

In \cite{VSh3} the conjecture was confirmed for ``even'' ladders (see Definition \ref{even} below). For such ladders the following theorem was proved.
\begin{theorem}\label{Th1}
The coefficient $s$ from the asymptotic \eqref{eq:mes_asymp} satisfies the relation
\[
	\forall t\in [0,\nu]\quad s(t)=e^{-Dt}\,\sigma(t),
\]
where $\sigma$ is some purely singular non-decreasing function.
\end{theorem}
Hence the relation $s(t)\neq const$ follows immediately.
This result is generalized in \cite{V3} to the case of 
the fourth order equations.

The aim of this paper is to generalize this result to a much broader class of ladders.

This paper has the following structure. \refpar{2} is a review, that provides the necessary definitions of self-similar functions of generalized Cantor type, derives their properties and defines the classes of functions under consideration. \refpar{3}
establishes the spectral periodicity for the Neumann problem, similar to \cite{VSh3}, and also a weaker variant of spectral ``quasiperiodicity'' for some other boundary value problems. Finally, in \refpar{4} Theorem \ref{Th1} is proved for the suggested class of ladders.

\section{Self-similar functions of generalized Cantor type}\label{par:2}
\noindent
Let $m\geq 2$, and let
$\{I_k = [a_k,b_k]\}_{k=1}^m$ be the sub-segments of $[0,1]$, without interior intersections. Denote by
$S_k(t) = a_k + (b_k-a_k)\,t$ the affine contractions of $[0,1]$ onto
$I_k$ keeping the orientation. We also introduce the set of positive values
$\{\rho_k\}_{k=1}^m$ such that
$\sum\limits_{k=1}^m\rho_k = 1$.

We define the operator $\mathcal{S}$ on the space $L_{\infty}(0,1)$ as follows:
\begin{equation*}
\mathcal{S}(f) = \sum\limits_{k=1}^m
\left(\chi_{I_k}(f\circ S_k^{-1})+\chi_{\{x>b_k\}} \right)\rho_k.
\end{equation*}

\begin{proposition}
\textbf{\textsc{(see, e.g. \cite[Lemma 2.1]{Sh})}}
$\mathcal{S}$ is a contractive mapping on $L_{\infty}(0,1)$.
\end{proposition}
Hence, by the Banach fixed-point theorem there exists a (unique) function $C\in L_{\infty}(0,1)$ such that $\mathcal{S}(C)=C$.
\begin{definition}
Such a function $C(t)$ will be called the \textit{generalized Cantor ladder} with $m$ steps.
\end{definition}
The function $C(t)$ can be found as the uniform limit of a sequence $\mathcal{S}^k(f)$ for \mbox{$f(t)\equiv t$}, which allows us to assume that it is continuous and monotone, and also $C(0)=0$, $C(1)=1$.
The derivative of the function $C(t)$ in the sense of distributions is a singular measure $\rho$ without atoms, self-similar in the sense of Hutchinson (see \cite{H}), i.e. it satisfies the relation
\begin{equation*}
\rho(E) = \sum\limits_{k=1}^m \rho_k \cdot \rho(S_k^{-1}(E\cap I_k)).
\end{equation*}
More general ways to construct self-similar functions are described in
\cite{Sh}.

\begin{definition}\label{arithm}
The self-similarity is called \textit{arithmetic} if
the logarithms of the values $\rho_k(b_k-a_k)$ are commensurable.
\end{definition}

\begin{definition}\label{even}
We call the generalized Cantor ladder \textit{even} if
\begin{equation*}
 \forall k=2,\dots, m\quad \rho_k=\rho_1=\frac{1}{m},
\quad b_k - a_k=b_1 - a_1,\quad a_k - b_{k-1}=a_2 - b_1.
\end{equation*}
This class of ladders is considered in \cite{VSh3}.
\end{definition}

In this paper the results of \cite{VSh3} are generalized to the following class of functions. We assume that $a_1=0$, $b_m=1$. We require that all intermediate segments have non-empty interior
($a_k-b_{k-1} \neq 0$), but do not impose any other restrictions on their size. In addition, we require the values $\rho_k(b_k-a_k)$ to be equal, i.e.
\begin{equation}\label{eq:2.1} \forall k = 1,\dots, m \quad \rho_k(b_k-a_k) = \tau
\end{equation} for some constant $\tau$. Barring this condition the values of $\rho_k$ and $b_k-a_k$ can be arbitrary.

\begin{remark}
For this class of functions the following relations between parameters hold:
\begin{equation}\label{period}
\nu = -\ln\tau, \quad \tau^{-D} = m.
\end{equation}
\end{remark}

\section{Spectral periodicity}\label{par:3}
We consider the formal boundary value problem
\begin{gather}\label{eq:3.1}
    -y''-\lambda\rho y=0,\\ \label{eq:3.2}
    y'(0)-\gamma_0 y(0)=y'(1)+\gamma_1 y(1)=0.
\end{gather}
We call the function $y \in W_2^1[0,1]$ its generalized solution if it satisfies the integral identity
$$ \int\limits_0^1 y'\eta'\, dx +
\gamma_0y(0)\eta(0)+\gamma_1y(1)\eta(1) = 
\lda\int_0^1 y\eta \;\rho(dx)  $$
for any $\eta \in W_2^1[0,1]$. Substituting functions \mbox{$\eta \in \Wo_2^1[0,1]$} into the integral identity, we establish that the derivative $y'$ is a primitive of a singular measure without atoms $\lda\rho y$, whence $y \in C^1[0,1]$. In addition, we need the following oscillation properties of the eigenfunctions.

\begin{proposition}\label{oscil} 
\textbf{\textsc{(\cite[Proposition~11]{V2})}}
Let $\{\lambda_n\}_{n=0}^{\infty}$ be a sequence
of eigenvalues of the boundary value
problem \eqref{eq:3.1}, \eqref{eq:3.2} numbered in ascending order. Then, regardless of the choice of index $n\in\mathbb N$, eigenvalue $\lambda_n$ is
simple, and corresponding eigenfunction does not vanish on the boundary of the segment \([0,1]\) and has exactly $n$ different zeros within this segment.
\end{proposition}

We now prove the main statements of this section.

\begin{theorem}\label{4.1}
Let $\{\lambda_n\}_{n=0}^{\infty}$ be a sequence of eigenvalues of the problem
\eqref{eq:1.1}, \eqref{eq:1.2} numbered in ascending order.
Then, regardless of the choice of index $n\in\mathbb N$, the following equality holds:
\begin{equation}\label{eq:4.1}
	\tau\lambda_{m n}=\lambda_n.
\end{equation}
\end{theorem}
\begin{proof}
The scheme of the proof repeats \cite[3.1.1]{VSh3}. Let's fix the eigenfunction $y_n$ corresponding to the eigenvalue $\lambda_n$. We construct a function $z\in C[0,1]$  satisfying the following conditions:
\begin{equation*}
z = c_k\cdot(y_n\circ S_k^{-1})\quad\text{ on }\quad I_k,
\end{equation*}
in addition, we continue it with constants on intermediate segments.
Non-zero values $c_k$ are chosen in such a way that the values of the function $z$ coincide on the ends of intermediate segments. It is easy to see that the resulting function is
continuously differentiable and satisfies the equality
$$z'(0)=z'(1)=0.$$
It is easy to check using \eqref{eq:2.1} that $z$ is an eigenfunction of the boundary value problem \eqref{eq:1.1}, \eqref{eq:1.2}, corresponding to eigenvalue $\tau^{-1}\lda_n$. In addition, it has exactly $mn$ zeros on the interval $[0,1]$ which leads to the equality \eqref{eq:4.1} (see Proposition \ref{oscil}).
\end{proof}

In \cite[3.1.2]{VSh3} the relation was also obtained between the eigenvalues with the numbers $n$ and $m(n+1)-1$ in certain problems with mixed boundary conditions of the form \eqref{eq:3.1}, \eqref{eq:3.2}. The same relation cannot be proved  in a more general case, so we'll derive a one-sided estimate, which we call
{\it spectral quasiperiodicity}.

Let's fix the eigenfunction $y_n$, corresponding to the eigenvalue $\lda_n$ of the problem with boundary condition
\[
	y'(0)-\gamma^{(1)}y(0)=y'(1)+\gamma^{(1)}y(1)=0.
\]
 
We construct the function $z$ as follows. Define
\[
	z = c_k\cdot(y_n\circ S_k^{-1}) 
	\quad\text{ on }\quad I_k,
\]
in addition, we will continue it smoothly with linear functions on intermediate segments till the intersection with the $x$-axis. If we set $\gamma^{(1)} = \max\limits_{k} \big(\frac{2}{|a_{k+1}-b_k|}\big)$, intersections will be close to the edges of $I_k$, and the function will remain undefined in the middle of intermediate segments. We will define it as zero on all remaining intervals.
We define the signs of non-zero parameters $c_k$ so that at each intermediate segment function $z$ has a change of sign. By \eqref{eq:2.1}, the resulting function satisfies the equation \eqref{eq:1.1} for $\lda = \tau^{-1}\lda_n$ almost everywhere. Unfortunately, it is not smooth.

We will make some continuous transformation with it which will not increase the value $\lda$ and will not change the number of changes of sign. As a result, we get a smooth eigenfunction of a certain boundary value problem and write the estimate of eigenvalues of this boundary value problem through $\lda_n$.

Our transformation consists of several steps. At the step $j$ the function $z$ is composed of eigenfunctions of boundary value problems on subsections $I_k$ with certain boundary conditions
\[
	z'(a_k)-\alpha_j^{(k)}z(a_k)=z'(b_k)+\beta_j^{(k)}z(b_k)=0
\]
and is continued linearly on intermediate segments. $z$ is smooth on some intermediate sections and piecewise linear on others.
We fix
$\alpha_j^{(1)}$ and $\beta_j^{(m)}$ and continuously change the rest of the values $\alpha_j^{(k)}$ and $\beta_j^{(k)}$ so that the eigenvalues, which correspond to the functions on segments $I_k$ remain equal, $z$ remains smooth on the intervals where the smoothness has already been achieved, and $\beta_j^{(k)}$, $\alpha_j^{(k+1)}$ decrease at the ends of the rest of the segments. This procedure decreases the value of $\lda$ by the variational principle and does not change the number of changes of sign according to Proposition \ref{oscil}.

At some point, at least on one intermediate segment zero interval of the function $z$ will collapse into one point. Assume that it occurred between the segments $I_l$ and $I_{l+1}$. At this point, we multiply $c_{l+1}$ and all subsequent coefficients by the same ratio so that $z$ becomes smooth on the interval $[a_l, b_{l+1}]$.

After $m-1$ step $z$ becomes completely smooth. After that we can reduce the parameters of boundary conditions at the ends so that
\[
	\alpha_m^{(1)} = \beta_m^{(m)} = \gamma^{(2)} := \gamma^{(1)}\cdot\min \{|I_1|, |I_m|\}.
\]
Note that the constructed function $z$ has exactly $m(n+1)-1$ roots and, hence, is an eigenfunction corresponding to the eigenvalue $\mu_{m(n+1)-1}$ of the boundary value problem
\[
	z'(0)-\gamma^{(2)}z(0)=z'(1)+\gamma^{(2)}z(1) = 0.
\]
Moreover, by construction, the resulting $\mu_{m(n+1)-1}$ does not exceed the original value of $\tau^{-1}\lda_n$. Thus, we have proved the following statement.

\begin{theorem}\label{4.2}
Let 
$\gamma^{(1)} = \max\limits_{k} \big(\frac{2}{|a_{k+1}-b_k|}\big)$,
$\gamma^{(2)} = \gamma^{(1)}\cdot\min \{|I_1|, |I_m|\}$.
Let $\{\lambda_n\}_{n=0}^{\infty}$ be the sequence of eigenvalues numbered in ascending order corresponding to the boundary value problem 
\[
	y'(0)-\gamma^{(1)}y(0)=y'(1)+\gamma^{(1)}y(1)=0
\]
for the equation \eqref{eq:3.1}.
Denote by $\{\mu_n\}_{n=0}^{\infty}$ a similar sequence corresponding to the boundary value problem
\[
	y'(0)-\gamma^{(2)}y(0)=y'(1)+\gamma^{(2)}y(1)=0
\]
for the same equation.
Then, regardless of the choice of index $n\in\mathbb N$, the following inequality holds:
\[
	\tau\mu_{m (n+1)-1}\leq\lda_n.
\]
\end{theorem}

\section{Specification of the spectrum characteristics}\label{par:4}
To prove Theorem \ref{Th1} we use the following facts:

\begin{proposition}\label{sing}
\textbf{\textsc{(\cite[Proposition~4.1.3]{VSh3})}}
Suppose that $f\in L_2[0,1]$ is a bounded non-decreasing function,
$\{f_n\}_{n=0}^{\infty}$ is a sequence of non-decreasing step functions and $\{\mathfrak A_n\}_{n=0}^{\infty}$ is the sequence of  discontinuity points of functions $f_n$. Suppose also that the following asymptotic relation holds as  $n\to\infty$:
\[
	(\#\mathfrak A_n+2)\cdot\|f-f_n\|_{L_2[0,1]}=o(1).
\]
Then the monotone function $f$ is purely singular.
\end{proposition}

\begin{proposition}\label{5.2}
\textbf{\textsc{(\cite[Proposition~5.2.1]{VSh3})}}
Let $\{\lambda_n\}_{n=0}^{\infty}$ be a sequence of the eigenvalues of boundary value problem
\begin{gather*}
	-y''-\lambda\rho y=0,\\ 
	y'(0)=y'(1)=0,
\end{gather*}
numbered in ascending order.
Let $\{\mu_n\}_{n=0}^{\infty}$ be a similar sequence corresponding to the boundary value problem
\[
	y'(0)-\gamma_0y(0)=y'(1)+\gamma_1y(1)=0
\]
with $\gamma_0,\,\gamma_1\geq 0$ for the same equation. Then
\[
	\sum\limits_{n=1}^{\infty} |\ln\mu_n-\ln\lambda_n|<+\infty.
\]
\end{proposition}

\begin{remark}
More general results about regularized products of eigenvalues were also considered in \cite{Naz2}.
\end{remark}

{\it Proof of Theorem \ref{Th1}:}

Let $\{\lambda_n\}_{n=0}^{\infty}$, $\{\mu^{(1)}_n\}_{n=0}^{\infty}$ and $\{\mu^{(2)}_n\}_{n=0}^{\infty}$ be sequences of eigenvalues of equation \eqref{eq:1.1} numbered in ascending order and corresponding to the boundary value problems
\begin{align*}
	\lambda_n:&\quad y'(0)=y'(1)=0,\\
	\mu^{(1)}_n:&\quad y'(0)-\gamma^{(1)}y(0)=y'(1)+\gamma^{(1)}y(1)=0,\\
	\mu^{(2)}_n:&\quad y'(0)-\gamma^{(2)}y(0)=y'(1)+\gamma^{(2)}y(1)=0,
\end{align*}
where the parameters $\gamma^{(1)}$ and $\gamma^{(2)}$ were introduced in Theorem
\ref{4.2}.

We set $\sigma_k(t) := m^{-k}N(e^{k\nu+t})$, where
$N$ is a counting function for $\lambda_n$.
By the relation \eqref{eq:mes_asymp} for all $t\in [0,\nu]$ we have the equality
$\sigma(t)=\lim\limits_{k\to\infty}\sigma_k(t)$. Note also that $\sigma_k$ and $\sigma_{k+1}$ differ by no more than $m^{-k}$ for all $t\in [0,\nu]$. 

In view of the equalities \eqref{eq:4.1} and \eqref{period}, regardless of the choice of index $k\in\mathbb N$, values of the functions $\sigma_k(t)$ and $\sigma_{k+1}(t)$ coincide for all $t\in [0,\nu]$, satisfying for some
$n\in\mathbb N$ the inequality
\begin{equation}\label{eq:5.1}
	\lambda_{m (n+1)-1} < 
	e^{(k+1)\nu+t} < \lambda_{m (n+1)}.
\end{equation}
Let's estimate the measure of the set of all other points $t$. If $t \in [0, \nu]$ does not admit \eqref{eq:5.1}, then the following relation holds:
\[
(k+1)\nu + t \in \Big( 
\bigcup\limits_{n=0}^\infty 
\left[\ln\lda_{mn}, \ln\lda_{m(n+1)-1}\right]
\Big) \cap [(k+1)\nu, (k+2)\nu].
\]
Let's estimate the sequence of partial sums of the series
\[
	\sum\limits_{n=1}^{\infty}|\ln\lambda_{m (n+1)-1}-
		\ln\lambda_{m n}| \leq 
		\sum\limits_{n=1}^{\infty}|\ln\lambda_{m (n+1)-1}-
		\ln\mu^{(2)}_{m (n+1)-1}| +
		\sum\limits_{n=1}^{\infty}|\ln\mu^{(2)}_{m (n+1)-1}-
		\ln\lambda_{mn}|.
\]
We estimate these sums separately.
\[
	\sum\limits_{n=1}^{\infty}|\ln\lambda_{m (n+1)-1}-
		\ln\mu^{(2)}_{m (n+1)-1}| \leq
		\sum\limits_{n=1}^{\infty}|\ln\lambda_{n}-
		\ln\mu^{(2)}_{n}| \leq C.
\]
Here the first inequality is obtained by expanding the set of terms. The second one
is an application of the Proposition \ref{5.2}.
\begin{align*}
	\sum\limits_{n=1}^{\infty}|\ln\mu^{(2)}_{m (n+1)-1}-
	\ln\lambda_{mn}| &=
	\sum\limits_{n=1}^{\infty}\ln\mu^{(2)}_{m (n+1)-1}-
	\ln\lambda_{mn} \\
	&\leq \sum\limits_{n=1}^{\infty}\ln\mu^{(1)}_{n}-
	\ln\lambda_{n} =
	\sum\limits_{n=1}^{\infty}|\ln\mu^{(1)}_{n}-
	\ln\lambda_{n}| \leq C.
\end{align*}
Here the first inequality follows from Theorems \ref{4.1} and \ref{4.2}, the second inequality follows from Proposition \ref{5.2}, while the equalities follow from the estimates
\[
	\mu^{(2)}_{m(n+1)-1} > \mu^{(2)}_{mn} > \lda_{mn}, \quad
	\mu^{(1)}_n > \lda_n.
\]
Thus we have shown that the measure of the set
$\bigcup\limits_{n=0}^\infty 
\left[\ln\lda_{mn}, \ln\lda_{m(n+1)-1}\right]$
is bounded, which means that after intersecting with segments \mbox{$[(k+1)\nu, (k+2)\nu]$} we obtain
\[
	\mathrm{meas}\,\{t\in [0,\nu]\::\:\sigma_{k+1}(t)\neq\sigma_k(t)\} =
	o(1), \quad k\to\infty.
\]
Consequently, the following estimate holds:
\[\|\sigma_{k+1}-\sigma_k\|_{L_2[0,\nu]}=o(m^{-k}),\]
and so does the asymptotics
\[\|\sigma_k-\sigma\|_{L_2[0,\nu]}=o(m^{-k}),\]
derived therefrom.
Let's make sure that the number of discontinuity points of the function $\sigma_k$ admits the estimate $O(m^k)$ as $k\to\infty$. Using the relation \eqref{eq:4.1}, we obtain the following inequality:
\[
m^{k+c} + 1 = N(\lda_{m^{k+c}}) = N(\tau^{-(k+c)}\lda_1) >
N(e^{k\nu+t}),
\]
where $c > \nu^{-1}(1-\ln \lda_1)$ is an integer. It remains to note that the number of discontinuity points of the function $N(\lda)$ on the segment does not exceed its value on the right end.

Thus, the function $\sigma$ along with a sequence of piecewise constant approximations $\sigma_k$ satisfy all of the conditions of the Proposition \ref{sing}, which proves the theorem.

\bigskip
\bigskip

This research is supported by the Chebyshev Laboratory  (Department of Mathematics and Mechanics, St. Petersburg State University) under RF Government grant 11.G34.31.0026, by Russian Foundation for Basic Research (project 13-01-00172A), by St. Petersburg State University grant N6.38.64.2012 and by JSC "Gazprom Neft".

\bigskip
\bigskip

\begin{enbibliography}{99}
\addcontentsline{toc}{section}{References}

\bibitem{SV} Solomyak~M., Verbitsky~E. \emph{On a spectral problem related to
self-similar measures}
// Bull. London Math.~Soc. --- 1995. --- V.~27, N~3. ---
P.~242-248.

\bibitem{VSh3} Vladimirov~A.~A., Sheipak~I.~A. \emph{On the Neumann Problem for the Sturm--Liouville Equation with Cantor-Type Self-Similar Weight}// Functional Analysis and Its Applications. --- 2013. --- V.~47, N.~4. --- P.~261--270.

\bibitem{BS2} Birman~M.~Sh., Solomyak~M.~Z. \emph{Spectral theory of self-adjoint operators in Hilbert space} // 
Ed.2, Lan' publishers. --- 2010. (in Russian); English translation of the 1st ed: Mathematics and its Applications (Soviet Series). D. Reidel Publishing Co., Dordrecht, 1987.

\bibitem{K} Krein~M.~G. \emph{Determination of the density of the symmetric inhomogeneous string by spectrum}//
Dokl. Akad. Nauk SSSR --- 1951. --- V.~76, N.~3. --- P.~345-348. (in Russian)

\bibitem{BS1} Birman~M.~Sh., Solomyak~M.~Z. \emph{Asymptotic behavior of the spectrum of weakly polar integral operators}// Mathematics of the USSR-Izvestiya. --- 1970. --- V.~4, N.~5. --- P.~1151-1168.

\bibitem{B} Borzov~V.~V. \emph{On the quantitative characteristics of singular measures}// Problems of math. physics. --- 1970. --- V.~4.
--- P.~42-47. (in Russian)

\bibitem{F} Fujita~T. \emph{A fractional dimention, self-similarity
and a generalized diffusion operator} // Taniguchi Symp. PMMP. Katata.
--- 1985. --- P.~83-90.

\bibitem{HU} Hong~I., Uno~T. \emph{Some consideration of asymptotic 
distribution of eigenvalues for
the equation $d^2u/dx^2 + \lda\rho(x)u = 0$}// Japanese Journ. of Math.
--- 1959. --- V.~29. --- P.~152-164.

\bibitem{MR} McKean~H.~P., Ray~D.~B. \emph{Spectral distribution of a differential operator}// Duke Math. Journ. --- 1962. ---
V.~29. --- P.~281-292.

\bibitem{KL} J.~Kigami, M.~L.~Lapidus. \emph{Weyl's problem for the spectral 
distributions of Laplacians on p.c.f. self-similar fractals}// Comm. Math. Phys. ---
1991. --- V.~158. --- P.~93-125.

\bibitem{Naz} Nazarov~A.~I. \emph{Logarithmic \(L_2\)-small ball asymptotics with respect to self-similar measure for some Gaussian processes}// Journal of Mathematical Sciences (New York). --- 2006. --- V.~133, N.~3. --- P.~1314-1327.

\bibitem{VSh1} Vladimirov~A.~A., Sheipak~I.~A. \emph{Self-similar functions in \(L_2[0,1]\) and the Sturm--Liouville problem with singular indefinite weight}// Sbornik: Mathematics. --- 2006. --- V.~197, N.~11. --- P.~1569-1586.

\bibitem{V3} Vladimirov~A.~A. \emph{Method of oscillation and spectral problem for four-order differential operator with self-similar weight}// \texttt{arxiv:1107.4791} (in Russian)

\bibitem{Sh} Sheipak~I.~A. \emph{On the construction and some properties of self-similar functions in the spaces \(L_p[0,1]\)}//  Mathematical Notes. --- 2007. ---
 V.~81, N.~6. --- P.~827-839.

\bibitem{H} Hutchinson~J.~E. \emph{Fractals and self similarity}//
Indiana Univ. Math. J. --- 1981. --- V.~30, N~5. --- P.~713-747.

\bibitem{V2} Vladimirov~A.~A. \emph{On the oscillation theory of the Sturm--Liouville problem with singular coefficients}// Computational Mathematics and Mathematical Physics. --- 2009. --- V.~49, N.~9. --- P.~1535-1546.

\bibitem{Naz2} Nazarov~A.~I. \emph{On a set of transformations of Gaussian random functions}// Theory of Probability and its Applications. --- 2010. --- 
V.~54, N.~2. --- P.~203-216.

\end{enbibliography}

\end{document}